\documentclass[11pt]{amsart}
\usepackage{amsmath}
\usepackage{amssymb}
\usepackage{amscd}

\usepackage{url}
\usepackage[all]{xypic}

\usepackage{graphicx}

\usepackage{anysize}

\numberwithin{equation}{section}

\theoremstyle{plain}
\newtheorem{theorem}[equation]{Theorem}
\newtheorem{lemma}[equation]{Lemma}
\newtheorem{proposition}[equation]{Proposition}

\newtheorem{corollary}[equation]{Corollary}

\theoremstyle{definition}

\newtheorem{remark}[equation]{Remark}

\newtheorem{definition}[equation]{Definition}

\newtheorem*{acknowledgements}{Acknowledgements}

\newcommand{\sss}{\ifmmode{{\mathfrak s}}\else{${\mathfrak s}$\ }\fi}
\newcommand{\sst}{\ifmmode{{\mathfrak t}}\else{${\mathfrak t}$\ }\fi}
\def\Z{\mathbb Z}

\def\Q{\mathbb Q}
\def\C{\mathbb C}

\def\mm{\mathbf{m}}
\def\kk{\mathbf{k}}
\def\wt#1{\widetilde{#1}}

\def\spinc{spin$^c$}

\title[Rational cuspidal curves in projective surfaces]{Rational cuspidal curves in projective surfaces. Topological versus algebraic obstructions}
\author{Maciej Borodzik}
\address{Institute of Mathematics, University of Warsaw, ul. Banacha 2,
02-097 Warsaw, Poland}
\address{Institute of Mathematics, Polish Academy of Sciences, ul. \'Sniadeckich 8, Warsaw, Poland}
\email{mcboro@mimuw.edu.pl}
\thanks{The first author was supported by  Polish OPUS grant No 2012/05/B/ST1/03195}

\date{\today}
\makeatletter
\@namedef{subjclassname@2010}{\textup{2010} Mathematics Subject Classification}
\makeatother

\begin{document}
\begin{abstract}
We study rational cuspidal curves in  projective surfaces. We specify two criteria obstructing possible
configurations of singular points that may occur on such curves. 
One criterion  generalizes the result of Fernandez de Bobadilla, Luengo, Melle--Hernandez
and N\'emethi and is based on the B\'ezout theorem. The other one is a generalization of the result obtained by Livingston and the 
author and relies on Ozsv\'ath--Szab\'o inequalities for $d$-invariants in Heegaard Floer homology. 
We show by means of explicit calculations that the two approaches give very similar obstructions.
\end{abstract}
\maketitle

\section{Introduction}
Suppose $C$ is an algebraic curve in a smooth complex projective surface $X$. Throughout the paper we will always assume that $C$
is reduced and irreducible. A singular point $p\in C$ is called \emph{cuspidal} if $C$ is locally irreducible at $p$, that is,
if for a small ball $B\subset X$ with center $p$, the set $(B\setminus\{p\})\cap C$ is connected. A \emph{rational cuspidal curve}
in $X$ is a algebraic curve of genus $0$ all of whose singular points are cuspidal. Topologically, a rational cuspidal curve
is homeomorphic to a sphere $S^2$.

A fundamental question in the theory is: 
which configurations of singular points can occur on a single rational cuspidal curve? Recently the interest 
in the question has been renewed because of the break-through of Koras and Palka \cite{KP,Pal0,Pal}.
Most of the research, including the results of Koras and Palka,
has focused on the case $X=\C P^2$. In \cite{BM} Moe and the author studied rational cuspidal curves in Hirzebruch surfaces.

In this paper we turn our attention to rational cuspidal curves in projective surfaces in general. Note that there exist surfaces that do not 
contain
any rational curve, see for example \cite{Xu}, 
and it is conjectured that the a surface of general type has only finitely many rational curves
(a stronger version of this conjecture is known as the Green--Griffiths--Lang conjecture, see \cite{Voi}),
not to mention rational cuspidal curves. Constructing rational cuspidal curves on general surfaces is actually much more
challenging than obstructing their existence.

In the present article we give two obstructions, yet the most interesting aspect is probably the fact that two completely different methods,
one algebraic the other one topological,
give almost the same restrictions for possible configurations of singular points on rational cuspidal curves.
The main difference in the two versions are
the assumptions: the topological version requires that the geometric genus of the surface is zero, and works only for curves with a positive 
self-intersection.
However, unlike the algebraic version, it does not restrict to algebraic curves, but after a slight reformulation
the topological proof works even for curves in an almost complex four-manifold
with $b_2^+=1$; see Theorem~\ref{thm:maintop2}.

Despite of the differences of the assumptions, the similarity between Theorem~\ref{thm:mainalg} and Theorem~\ref{thm:maintop} is too strong to be
only a coincidence. A natural question, whether there is a theory unifying both the topological and algebraic approach, remains unanswered.

To state the main theorem(s) of the paper we introduce some notation (most of the terms are explained in Sections~\ref{sec:algnut} and \ref{sec:topnut}). 
Let $C$ be a cuspidal curve in a smooth projective surface $X$.
Let $z_1,\ldots,z_n$ be its singular points and $S_1,\ldots,S_n$ the corresponding semigroups. For any $m\in\Z$, $m\ge 0$ set
\begin{equation}\label{eq:defofR}
R(m)=\min_{\substack{m_1+\ldots+m_n=m\\m_i\ge 0}} \sum_{i=1}^n\# S_i\cap[0,m_i).
\end{equation}

The algebraic version of the main theorem is the following.
\begin{theorem}[Main Theorem, algebraic version]\label{thm:mainalg}
Suppose $D$ is a divisor on $X$ such that $C\cdot D\ge 0$. Then, if no section of $\mathcal{O}_X(D)$ vanishes
entirely on $C$
\[R(C\cdot D+1)\ge h^0(D).\]
\end{theorem}

The topological counterpart gives an almost the same estimate under different assumptions.
\begin{theorem}[Main Theorem, topological version]\label{thm:maintop}
Assume that $C$ is rational.
Suppose $D$ is a divisor on $X$ such that $C\cdot D+1\in[0,2g]$, where $g$ is the genus of the connected sum of links of singularities. 
If $p_g(X)=0$, $C^2>0$ and $K\cdot C\le 1$, then
\begin{equation}\label{eq:onRtop}
R(C\cdot D+1)\ge \chi(D)+\frac12 b_1(X).
\end{equation}
\end{theorem}

\begin{remark}\label{rem:aftertop}\
\begin{itemize}
\item[(a)] The algebraic version \emph{does not} assume that $C$ is rational. 
In the topological version, if $C$ has geometric genus $p_g(C)>0$, then an analogue of Theorem~\ref{thm:maintop} might be proved
using methods of \cite{BCG,BHL}, but in general we will obtain weaker inequalities.
\item[(b)] 
In Theorem~\ref{thm:maintop2} below, we will show a variant of Theorem~\ref{thm:maintop} working for smooth curves with cuspidal singularities
in four-manifolds admitting an almost complex structure.
\item[(c)] Theorem~\ref{thm:maintop} is formulated in the language of divisors, but in fact one can assume that $D$ is just a homology class
in $H_2(X;\Z)$. Indeed, as we assume that $p_g(X)=0$, we have $H^2(X;\C)=H^{1,1}(X)$ 
and by \cite[Theorem IV.2.13]{BH} each homology class in $H_2(X;\Z)$ can
be actually represented by a divisor. The only moment in the proof of Theorem~\ref{thm:maintop}, where algebraic geometry is used, is
replacing the expression $\frac12 D\cdot(D-K)+1$ by $\chi(D)+\frac12 b_1(X)$ in the last step of the proof.
\end{itemize}
\end{remark}

Theorems~\ref{thm:mainalg} and \ref{thm:maintop} are direct generalization of the main result of \cite{BM} for Hirzebruch surfaces.
For $X=\C P^2$ the inequality $R(C\cdot D+1)\ge h^0(D)$ turns out to be an equality; see \cite{BL}. Insofar only topological tools are
strong enough to prove the equality for the case of $\C P^2$. There are examples in \cite{BM} that the inequality is strict in the case
of Hirzebruch surfaces.
to prove q
\section{The algebraic proof}
\subsection{Projective surfaces in a nutshell}\label{sec:algnut}
For the reader's convenience we review some notions from algebraic geometry of surfaces. Good resources are \cite{BH,GH} and
\cite{GS} for more topological approach. We stress that in \cite{BH} the authors deal with complex surfaces in general, while
we restrict to projective surfaces. Our formulae might be less general than in \cite{BH}.

Let $X$ be a smooth projective surface. A \emph{divisor} $D$ on $X$ is a finite linear combination of algebraic curves on $X$
with coefficients in $\Z$. A divisor gives an element in $H_2(X;\Z)$, which we still denote by $D$. By Poincar\'e duality $D$ defines
also a class in $H^2(X;\Z)$.
Given the Hodge decomposition $H^2(X;\C)=H^{0,2}(X)\oplus H^{1,1}(X)\oplus H^{2,0}(X)$, the class of a divisor $D$ is
in $H^{1,1}(X)\cap H^2(X;\Z)$. Moreover, each element in that group can be represented by a divisor.

A divisor $D$ determines an algebraic line bundle $\mathcal{L}_D$. We have $c_1(\mathcal{L}_D)=D\in H^2(X;\Z)$.
Conversely, any algebraic line bundle is of the 
form $\mathcal{L}_D$ for some divisor $D$.
In particular, take a line bundle $\Lambda^2T^*X$, where the cotangent bundle $T^*X$ is regarded as a complex bundle of rank~2. 
The divisor $K$ satisfying $\mathcal{L}_K=\Lambda^2 T^*X$
is called the \emph{canonical divisor}. In $H^2(X;\Z)$, we have the relation $K=-c_1(X)$.

The sheaf of sections of $\mathcal{L}_D$ is denoted by $\mathcal{O}(D)$. 
Its \v{C}ech cohomology groups are $H^i(X,\mathcal{O}(D))$, $i=0,1,\ldots$. We set  $h^i(D)=\dim H^i(X,\mathcal{O}(D))$ and $\chi(D)=\sum (-1)^i h^i(D)$.
For dimensional reasons, $h^i(D)=0$ for $i>2$.
 
The following result can be found in \cite[Theorem I.5.3]{BH}.
\begin{theorem}[Serre duality]\label{thm:serre}
For any divisor $D$ we have $h^i(D)=h^{2-i}(K-D)$.
\end{theorem}

We will also need the Riemann--Roch theorem  for line bundles on surfaces, see \cite[Section I.5]{BH}
or \cite[Chapter IV.1]{GH}.
\begin{theorem}[Riemann--Roch theorem for surfaces]\label{thm:rrsurface}
Let $D$ be a divisor. We have
\[\chi(D)=\frac12D\cdot(D-K)+\chi(\mathcal{O}_X).\]
\end{theorem}
In the above theorem, $\mathcal{O}_X$ is the sheaf of regular functions. We have $h^i(\mathcal{O}_X)=\dim H^{0,i}(X)$ where $H^{p,q}(X)$ is 
the Dolbault cohomology. In particular $h^0(\mathcal{O}_X)=1$.
For a surface, $h^1(\mathcal{O}_X)$ is denoted $q(X)$ and called
the \emph{irregularity}. Furthermore $h^2(\mathcal{O}_X)=h^0(K)=p_g(X)$ is the \emph{genus} of $X$.

As is shown in \cite[Section IV.2]{BH}, the irregularity and genus are topological invariants. In particular we have the following result:
\begin{lemma}\label{lem:sig}
We have $b_2^+(X)=p_g(X)+1$. Here $b_2^+$ is the dimension of the maximal subspace of $H_2(X;\C)$ on which the intersection form is
positive definite.
\end{lemma}
Moreover we have the following result, which follows directly from the symmetry of the Hodge numbers; see  \cite[Section IV.2]{BH}.
\begin{lemma}\label{lem:irr}
The irregularity $q(X)$ is equal to $\frac12b_1(X)$.
\end{lemma}
We conclude this part by the Hirzebruch signature theorem; see \cite[Theorem I.3.1]{BH}.
\begin{theorem}\label{thm:hirz}
Let $\sigma$ be the signature of $X$ (that is, the signature of the intersection form on $H_2(X)$) and $\chi$ the (topological) Euler
characteristic. Then $K^2=3\sigma+2\chi$.
\end{theorem}
\begin{remark}\label{rem:hirz}
Theorem~\ref{thm:hirz} is actually a combination of two results. The harder one, and known as the Hirzebruch signature theorem for smooth manifolds, 
is that $\sigma(X)=\frac13p_1(X)$, where $p_1$ is the first Pontryagin number. This fact holds for an arbitrary smooth
four-manifold. The second one,  $p_1(X)=c_1^2(X)-2c_2(X)$, is a relation between the Pontryagin class and the Chern classes of a
complex bundle. In particular, Theorem~\ref{thm:hirz} holds for almost complex four-manifolds as well with $K$ understood as  $c_1(T^*X)$; see 
\cite[Section I.4]{GS}.

\end{remark}
\smallskip
Suppose $C$ is a germ of a plane curve singularity at $0\in\C^2$. Take a sphere $S\subset\C^2$ with center $0$
and sufficiently small radius. Let $L=C\cap S$. Then $L$ is a link in $S$.
It turns out, see \cite[Appendix to Chapter I]{EN}, that the isotopy class of $L$ depends only on the topological type of the singularity.
\begin{definition}
The link $L$ is called the \emph{link of singularity}.
\end{definition}
The number of components of $L$ is easily seen to be equal to the number of branches of $C$ at $0$. In particular, if the singularity
is cuspidal, then $L$ is a knot. Many other invariants of the singular point can be recovered from the link $L$. For example
we have the following result.
\begin{lemma}\label{lem:deltaandgenus}
If $L$ is a knot, then the Seifert genus of $L$ is the $\delta$-invariant of the singular point.
\end{lemma}
\begin{remark}\
\begin{itemize}
\item[(a)] Lemma~\ref{lem:deltaandgenus} can be regarded as a topological definition of the $\delta$-invariant. The original one is 
the dimension over $\C$ of
the skyscraper sheaf $\pi_*\mathcal{O}_{\wt{C}}/(\mathcal{O}_C)$, 
where $\pi\colon\wt{C}\to C$ is the normalization map; see \cite[Section II.11]{BH}.
\item[(b)] The Seifert genus of a link $L$, denoted $g(L)$, 
is the minimal genus of a smooth, compact, connected and oriented surface $\Sigma\in S^3$ such
that $\partial\Sigma=L$. Usually one calls the Seifert genus simply the genus. In the present paper, as the word `genus' has many different meanings,
we will stick to the notion `Seifert genus'.
\end{itemize}
\end{remark}
With a cuspidal singular point $z$ of a (germ of a)
plane curve $C$ we can associate a numerical semigroup $S$, described in great detail in \cite[Chapter 4]{Wall}. 
It is the set of all non-negative numbers that can
arise as local intersection indices at $z$, $(C\cdot D)_z$, where $D$ is a germ of a plane curve not 
containing $C$, however, $D$ maybe reduced and reducible.
The semigroup is also a topological invariant, that is, it depends only on the topological type of the singularity. More precisely
we have the following result.
\begin{lemma}\label{lem:semigroup}
Assume $S$ is the semigroup of a cuspidal singularity.
Let $P(t)=\sum t^{k}$, where the sum is taken over $k\in\Z_{\ge 0}\setminus S$. Then $P(1)=\delta$ and 
$1+(t-1)P(t)$ is the Alexander polynomial of the link of singularity. 
\end{lemma}
\begin{remark}
We point out that the term \emph{plane curve singularity} refers to a singularity of a germ of a curve.
Therefore for studying
properties of singular points of plane curves it does not matter
if the curve is globally embedded in $\C^2$, $\C P^2$ or in an arbitrary projective surface.
\end{remark}
\smallskip
We pass to a global description of curves in a projective surface. So let $C$ be a reduced and irreducible curve in $X$.
We can define two genera for $C$. The \emph{arithmetic genus}, denoted $p_a(C)$ is the expression
\begin{equation}\label{eq:arithmeticgenus}
p_a(C)=C\cdot(C+K).
\end{equation}
The arithmetic genus does not depend on actual topology of $C$, but only on its homology class. The \emph{geometric genus}, $p_g(C)$ is the genus
of the normalization of $C$. In other words $C$ has geometric genus $g$ if there exists a smooth closed Rieman surface $\Sigma$ of
genus $g$ that maps continuously onto $C$ and the map is one-to-one except at finite number of points of $\Sigma$.

The arithmetic genus and geometric genus are related by the following formula, attributed to Serre; see \cite[Section II.11]{BH}.
\begin{proposition}\label{prop:genusserre}
Suppose $C$ has singular points $z_1,\ldots,z_n$. Let $\delta_1,\ldots,\delta_n$ be $\delta$-invariants of $z_1,\ldots,z_n$, respectively.
Then
\[p_g(C)=p_a(C)-\sum_{i=1}^n\delta_i.\]
\end{proposition}
\subsection{Proof of Theorem~\ref{thm:mainalg}}
The following proof is a modification of the argument of \cite[Proposition 2]{FLMN}.
Let $D$ be a divisor on $X$. Let $H=\Gamma(X,\mathcal{O}_X(D))$
be the space of global sections of $\mathcal{O}(D)$, so $h^0(D)=\dim H$. For any vector $\kk=(k_1,\ldots,k_n)$, let 
\[H_{\kk}=\{f\in H\colon (f\cdot C)_{z_i}\ge k_i\textrm{ for all } i=1,\ldots,n\}.\]
Here $(f\cdot C)_{z_i}$ is the local intersection index of the zero locus of  $f$ and $C$ at $z_i$.
The space $H_{\kk}$ is a vector subspace of $H$. We have the following result, proved first in \cite{FLMN}. The
present formulation is taken from \cite[Lemma 3.17]{BM}.
\begin{lemma}\label{lem:codim}
The codimension of $H_{\kk}$ in $H$ is less or equal to $\sum_{i=1}^n \# S_i\cap[0,k_i)$.
\end{lemma}
Suppose for some choice $m_1,\ldots,m_n$ such that $\sum m_i=s+1$ the statement of Theorem~\ref{thm:mainalg}
does not hold, that is $\sum\# S_i\cap[0,m_i)<h^0(D)$.

It follows that if $\mm=(m_1,\ldots,m_n)$, then $H_{\mm}$ has positive dimension, so there
exists a non-zero section $f$ of $\mathcal{O}(D)$ such that $(f\cdot C)_{z_i}\ge k_i$ for all $i=1,\ldots,n$. 
The total intersection index of the zero locus of $f$ 
with $C$ is at least $\sum m_i>C\cdot D$.
This is impossible so $f$ must vanish entirely on $C$. But this
contradicts the assumption that no global section of $\mathcal{O}(D)$ vanishes on entirely $C$.

\subsection{Comparison of $h^0(D)$ and $\chi(D)$}
Theorem~\ref{thm:mainalg} has $h^0(D)$ on the right hand side, while Theorem~\ref{thm:maintop} has $\chi(D)$. We know that 
$\chi(D)=h^0(D)-h^1(D)+h^2(D)$. We have the following result.
\begin{lemma}
Assume that $C$ and $D$ satisfy the assumptions of Theorem~\ref{thm:mainalg}, that is $C$ is a curve and $D$ is a divisor
such that $C\cdot D\ge 0$ and no non-trivial section of $\mathcal{O}(D)$ vanishes on $C$.

Suppose additionally $K\cdot C<0$ and $C^2>0$. Then $\chi(D)\le h^0(D)$.
\end{lemma}
\begin{proof}
Observe that
\begin{equation}\label{eq:KminusD}
(K-D)\cdot C=K\cdot C-D\cdot C<0.
\end{equation}
We want to show that $h^2(D)=0$. By Serre duality, see Theorem~\ref{thm:serre} 
we have $h^2(D)=h^0(K-D)$. Suppose $h^0(K-D)>0$, so that there is a non-trivial global section
of $\mathcal{O}(K-D)$. If the section does not vanish entirely on $C$, then $(K-D)\cdot C\ge 0$, contrary to \eqref{eq:KminusD}.
So suppose there is a non-trivial global section of $\mathcal{O}(K-D)$ vanishing along $C$ up to order $n>0$. 
It follows that $K-D$ is equivalent to $nC+E$ for an effective divisor $E$ not containing $C$. In particular $E\cdot C\ge 0$.
Therefore $(K-D)\cdot C=(nC+E)\cdot C\ge 0$ by the assumptions of the lemma. This contradicts \eqref{eq:KminusD}.
\end{proof}

\begin{remark}
We point out that analogous assumptions ($K\cdot C\le 1$ and $C^2>0$) appear in the assumptions of the topological version of the main theorem,
Theorem~\ref{thm:maintop}.
\end{remark}

\begin{acknowledgements}
The author is very grateful to J\'ozsi Bodn\'ar and Marco Golla for many stimulating discussions and to Piotr Achinger for many comments about
the earlier version of the draft.
\end{acknowledgements}

\section{The topological proof}
\subsection{Heegaard Floer theory in a nutshell}\label{sec:topnut}
For the reader's convenience we revise a few facts from low-dimensional topology. A general reference for four-manifolds and \spinc{} structures
is \cite{GS}, while for Heegaard Floer homologies we refer to \cite{OS-introduction,OS-introduction2}.

Heegaard Floer theory associates with a three-manifold $Y$ equipped with a \spinc{} structure~$\sss$ four chain complexes
with additional filtration, well defined up to a filtered chain homotopy equivalence. The homologies of these complexes, $HF^+(Y,\sss)$,
$HF^\infty(Y,\sss)$, $HF^-(Y,\sss)$ and $\widehat{HF}(Y,\sss)$ are invariants of the pair $(Y,\sss)$. 
Suppose that $Y$ is a rational homology three-sphere, that is,
$H_*(Y,\Q)\cong H_*(S^3,\Q)$. Then $HF^-,HF^+$ and $HF^\infty$ have an absolute grading by rational numbers; see \cite{OS-absolute}.
Based on this grading one constructs a rational number $d(Y,\sss)$, called the $d$-invariant of Ozsv\'ath--Szab\'o.

The strength of the $d$--invariants relies on the following result, proved in \cite[Theorem 9.1]{OS-absolute}.
\begin{theorem}\label{thm:osabs}
Suppose $Y$ is a rational homology three-sphere, 
bounding a smooth four-manifold $W$, whose intersection form is negative definite. Suppose $\sss$ is
a \spinc{} structure on $Y$ and $\sst$ is a \spinc{} structure on $W$ extending $\sss$.
Then the following inequality for the $d$-invariant holds.
\begin{equation}\label{eq:dinvinequality}
d(Y,\sss)\ge \frac14\left(c_1^2(\sst)-3\sigma(W)-2\chi(W)\right).
\end{equation}
\end{theorem}
\begin{remark}\label{rem:notation}
The meaning of $c_1^2(\sst)$ is the following: a \spinc{} structure on a manifold $W$ determines a complex line bundle on $W$
(called the determinantial line bundle) and $c_1(\sst)\in H^2(W;\Z)$ is its first Chern class. The square $c_1^2(\sst)$ is the self-intersection
of $c_1(\sst)$ (as $\partial W$ is a rational homology sphere, 
the intersection form on cohomology $H^2(W;\Q)\times H^2(W;\Q)\to\Q$ is well-defined). Usually $c_1^2(\sst)$ is a rational number,
even though $c_1(\sst)\in H^2(X;\Z)$.
\end{remark}

To make this result applicable one needs to calculate $d$-invariants. Fortunately, if $Y$ is obtained by a surgery on an algebraic knot 
(or on a connected sum of algebraic knots) and the surgery coefficient is greater than twice the Seifert genus of the knot, then $d$ can be algorithmically
computed from the Alexander polynomial of an algebraic knot (or from the Alexander polynomials of the summands for a connected sum of algebraic
knots). As the Alexander polynomial of an algebraic knot is tightly related to the semigroup, we can calculate the $d$-invariant
for surgeries in terms of semigroups. First we need to describe \spinc{} structures on a surgery manifold.

\begin{definition}
Suppose $Y=S^3_q(J)$ is the result of a surgery on a knot $J$ with coefficient $q$. Realize $Y$ as a boundary of $N$, where
$N$ is a union of a four-ball $B$ and a two-handle attached to $S^3=\partial B$ along the knot $J$ with
framing $q$. Denote by $F$ the surface in $N$ obtained by capping the core of the two-handle with a Seifert surface for $J$.
Then $F$ generates $H_2(N;\Z)$. With this notation, $Y$
has the following 
enumeration of \spinc{} structures:

Let $m\in[-q/2,q/2)$ be an integer. The \spinc{} structure $\sss_m$ on $Y$ is the unique \spinc{} structure on $Y$ that extends to a \spinc{}
structure on $N$, denoted $\sst_m$, such that $\langle c_1(\sst_m),F\rangle+q=2m$.
\end{definition}

The following formula for $d$-invariants of a large surgery on a connected sum of algebraic knots was for long time well-known to the
experts (it can be traced back to \cite[Section 8.1]{OS-absolute}). The present
formulation in the language of semigroups is
taken from \cite[Proposition 3.7]{BM}; see \cite{BL} for proofs.
\begin{theorem}\label{thm:compute}
Let $J=J_1\#\ldots\#J_n$ be a connected sum of algebraic knots, that is, knots that are links of cuspidal singularities. 
Let $S_1,\ldots,S_n$ be the corresponding semigroups and set $R(m)$ as in \eqref{eq:defofR} above.

Let $g(J)$ be the Seifert genus of $J$ and suppose $Y=S^3_q(J)$ is a surgery on $J$ with a coefficient $q>2g(J)$.
The $d$-invariant of Ozsv\'ath--Szab\'o of $(Y,\sss_m)$ is equal to
\[d(Y,\sss_m)=\frac{(q-2m)^2-q}{4q}-2(R(m+g)-m).\]
\end{theorem}

For future reference we give a partial explanation of the term $\frac{(q-2m)^2-q}{4q}$.
\begin{proposition}\label{prop:c1square}
We have $c_1^2(\sst_m)=\frac{(q-2m)^2}{q}$.
\end{proposition}
\begin{proof}
By construction we have $H_2(N;\Z)=\Z$ is generated by the class of $C$ and $H^2(N;\Z)=\Z$. Let $\alpha$ be the generator of $H^2(N;\Z)$ dual
to $C$, that is $\langle \alpha,C\rangle=1$. 
As $C^2=q$ and the intersection form on the cohomology is the inverse of the intersection form on the homology, we infer that
$\alpha^2=q^{-1}$.
The first Chern class $c_1(\sst_m)$ evaluates on $C$ to $2m-q$. Therefore $c_1(\sst_m)=(2m-q)\alpha$ and so 
the self-intersection
of $c_1(\sst_m)$ is equal to $(2m-q)^2\alpha^2=\frac{(q-2m)^2}{q}$.
\end{proof}

\subsection{Proof of Theorem~\ref{thm:maintop}}
Let $N$ be a tubular neighbourhood of $C$ in $X$. Set $Y=\partial N$. Then as in \cite[Section 2]{BL} we show that 
$N$ is the result of attaching a two-handle along a knot $J$ with framing $q$: here
$J=J_1\#\ldots\# J_n$, where $J_i$ is the link of the singularity $z_i$ and the framing $q=C^2$ is the self-intersection of the
rational cuspidal curve $C$. In particular $Y=S^3_q(J)$ is a $q$~surgery on $J$.

To apply Theorem~\ref{thm:compute}
we need to investigate the relation between the Seifert genus of $J$
and the surgery coefficient $q$.
\begin{lemma}\label{lem:arith}
The Seifert genus of the knot $J$ is equal to $\frac12C\cdot(K+C)-1$.
\end{lemma}
\begin{proof}
By Proposition~\ref{prop:genusserre} we have $p_g(C)=p_a(C)-\sum\delta_i$. The genus formula \eqref{eq:arithmeticgenus}
implies that $p_a(C)=C\cdot (C+K)$. On the other hand $\delta_i=g(J_i)$ by Lemma~\ref{lem:deltaandgenus},
so $\sum\delta_i=\sum g(J_i)=g(J)$ by additivity of the Seifert genus of the knot. As $C$ is rational, $p_g(C)=0$.  The formula follows.
\end{proof}
\begin{corollary}\label{cor:surgeryislarge}
The surgery coefficient $q$ is greater than twice the Seifert genus of $J$.
\end{corollary}
\begin{proof}
We have $q=C^2$ and $2g(J)=C\cdot(K+C)-2$. As $C\cdot K\le 1$, the corollary follows.
\end{proof}
Given the corollary we are allowed to use Theorem~\ref{thm:compute} to calculate $d$-invariants of $Y$ in terms
of the semigroups of singular points. Theorem~\ref{thm:osabs} will be used 
to bound the $d$ invariants.
In the present situation, the three-manifold is $Y=-\partial N$ 
and the four-manifold $W=X\setminus N$ is the complement of the neighborhood
of $C$ in the algebraic surface.
\begin{lemma}\label{lem:Wisnegative}
The manifold $W$ is has negative definite intersection form.
\end{lemma}
\begin{proof}
Consider the long exact sequence of the pair $(X,W)$. We have $H_3(X,W;\C)\to H_2(W;\C)\to H_2(X;\C)\to H_2(X,W;\C)\to H_1(W;\C)$.
Now $H_*(X,W)\cong H_*(N,\partial N)$ by excision, and $N$ is a tubular neighborhood of $C$, so (at least homologically) it can be regarded
as a disk bundle over $S^2$ and $\partial N$ is the associated sphere bundle. By Thom isomorphism $H_3(N,\partial N)=0$ and $H_2(N,\partial N)\cong\C$. 
Thus $H_2(W)$ injects into $H_2(X)$.

The map $H_2(X)\to H_2(X,W)$ can be described geometrically. A class $\alpha\in H_2(X)$ gets mapped to
$\alpha\cdot C$ times the generator of $H_2(X,W;\C)$. In particular, a class in $H_2(X)$ is in the image of $H_2(W)$ if and only if
it intersects $C$ trivially. Therefore $H_2(W;\C)$ is the orthogonal (with respect to the intersection form) complement in $H_2(X)$
of the one-dimensional space generated by the class of $C$. As the self-intersection of $C$ is positive, we infer that $b_2^+(W)=b_2^+(X)-1$.

By Lemma~\ref{lem:sig} and the assumption that $p_g(X)=0$ we infer that $b_2^+(X)=1$. Hence $b_2^+(W)=0$. This amounts to saying
that the intersection form is non-positive definite. It is negative definite because it is nondegenerate.
\end{proof}

The next lemma describes which \spinc{} structures on $Y$ extend over $W$.
\begin{lemma}\label{lem:spinc}
Let $D$ be a divisor on $X$. There is a unique \spinc{} structure on $X$, called $\sst_D$
whose first Chern class is $K+2D$. This \spinc{} structure restricts to a \spinc{} structure on $W$ still denoted by $\sst_D$.
Set $m=C\cdot D-g(J)+1$. If $m\in[-q/2,q/2)$, then $\sst_D$ extends to the \spinc{} structure $\sss_m$
on $Y$.

Moreover, if $\sss_m$ extends to a \spinc{} structure $\sst$ on $W$, then $\sst$ is of the form $\sst_D$ for some $D$.
\end{lemma}
\begin{proof}
Suppose $\sss_m$ extends to some \spinc{} structure on $W$. As it extends also to $\sst_m$ on $N$, we can glue these two \spinc{} structures
to obtain a global \spinc{} structure on $X$. As $X$ is a closed four-manifold, \spinc{} structures on $X$ are in one-to-one correspondence 
(via the first Chern class) with characteristics, that is, elements
$w\in H^2(X;\Z)$ such that $\langle w,x\rangle=x^2\bmod 2$, for every $x\in H_2(X;\Z)$. The canonical
divisor is a characteristic. Moreover, every characteristic is of form 
$K+2D$ for some divisor $D$ (for this we use the fact that any class in $H^2(X;\Z)$ can be represented by a divisor, see 
Remark~\ref{rem:aftertop}(c)). This shows that if $\sss_m$ extends over $W$, then the extension is the restriction to $W$ of some $\sst_D$.

Consider a \spinc{} structure $\sst_D$ on $X$. 
Observe that $(K+2D)\cdot C+C^2=2D\cdot C+2g-2$, where $g$ is the Seifert genus of the knot $J$. Set 
\begin{equation}\label{eq:mandD}
m=D\cdot C+g(J)-1.
\end{equation} 
Then the \spinc{} structure $\sst_D$ restricts to the \spinc{} structure $\sst_m$ on $N$ and hence to $\sss_m$ on $Y$.
\end{proof}
\begin{remark}\label{rem:orient}
In this discussion, the orientation on $Y$ is chosen so that $\partial N=Y$ and $\partial W=-Y$.
\end{remark}

We are now ready to complete the proof of Theorem~\ref{thm:maintop}. Write $g=g(J)$. Take a divisor $D'$ 
such that $m=D'\cdot C+g-1$ belongs to $[-q/2,q/2)$.
By Lemma~\ref{lem:spinc} 
$(-Y,\sss_m)$ bounds a negative definite manifold $W$ with a \spinc{} structure $\sst_{D'}$, so Theorem~\ref{thm:osabs} gives 
\[d(-Y,\sss_m)\ge \frac14\left(c_1^2(\sst_{D'})-3\sigma(W)-2\chi(W)\right).\]
By \cite[Proposition 4.2]{OS-absolute} the $d$-invariant changes sign if the orientation is changed. 
Rewriting $d$ as in Theorem~\ref{thm:compute} above we obtain
\[-\frac{(q-2m)^2-q}{4q}+2(R(m+g)-m)\ge \frac14\left(c_1^2(\sst_{D'})-3\sigma(W)-2\chi(W)\right).\]
From Proposition~\ref{prop:c1square} we have that
$\frac{(q-2m)^2}{q}=c_1^2(\sst_m)$. As $\sst_m$ and $\sst_{D'}$ are \spinc{} structures
that glue to the \spinc{} structure on $X=N\cup_Y W$ 
whose first Chern class is $K+2D'$, we infer that $c_1^2(\sst_m)+c_1^2(\sst_{D'})=(K+2D')^2$.
Therefore
\begin{equation}\label{eq:r1}
2R(m+g)-2m\ge \frac14(K+2D')^2-\frac34(\sigma(X)-1)-\frac24(\chi(X)-2).
\end{equation}
Indeed, $\sigma(W)+1=\sigma(X)$ and $\chi(W)+2=\chi(X)$. Then Hirzebruch signature theorem, Theorem~\ref{thm:hirz} allows us to rewrite the above inequality as
\[R(m+g)-m\ge \frac12D'\cdot(D'+K)+1.\]
We substitute $m=D'\cdot C+g-1$. Then
\begin{equation}\label{eq:onR}
R(C\cdot D'+C\cdot(C+K)+1)\ge \frac12 D'\cdot(D'+K)+C\cdot D'+\frac12 C\cdot(C+K)+1.
\end{equation}
Write $D=C+K+D'$. The above formula simplifies to
\[R(C\cdot D+1)\ge \frac12 D\cdot(D-K)+1.\]
By the Riemann--Roch theorem for line bundles on
surfaces, Theorem~\ref{thm:rrsurface}, we have that $\frac12 D\cdot(D-K)+\chi(\mathcal{O}_X)=\chi(D)$. 
But $\chi(\mathcal{O}_X)=1-\dim H^{0,1}(X)+\dim H^{0,2}(X)=1-q(X)+p_g(X)$, where $q(X)$ is the irregularity. Applying Lemma~\ref{lem:irr}
and using the assumption that $p_g(X)=0$ we obtain
\begin{equation}\label{eq:finalonR}
R(C\cdot D+1)\ge \chi(D)+\frac12 b_1(X).
\end{equation}

The requirement that $m\in[-q/2,q/2)$ translates into $C\cdot D+1\in[-q/2+g,q/2+g)$. As $q>2g$, the latter interval contains $[0,2g]$.
Therefore \eqref{eq:finalonR} if $C\cdot D+1\in[0,2g]$ and this is the statement of Theorem~\ref{thm:maintop}.

\subsection{Necessity of assumptions}

Let us go through the assumptions of Theorem~\ref{thm:maintop} and see to what extent have they been used.

The assumption that $C^2>0$ implies that the surgery on $J$ is positive, so that we can use Theorem~\ref{thm:compute}. If $C^2=0$, the three-manifold
$Y$ is not a rational homology sphere, but a variant of Theorem~\ref{thm:compute} can still be given. On the other hand, if $C^2<0$, then $Y$
is a surgery with negative coefficient. Negative surgeries on connected sums of L--space knots have very simple $d$-invariants. The $d$-invariant
is a sum of two terms, one similar to the one discussed in 
Proposition~\ref{prop:c1square}, that is, related to the surgery coefficient and the choice of  a \spinc{} structure. However the other one,
depends only on the Seifert genus; see for example \cite[Section 5]{Nem}.
We do not get anything more than the inequality for the Seifert genus if $C^2<0$.

\smallskip
The condition $C\cdot K\le 1$ was used to show that $C^2>2g(J)$ in Corollary~\ref{cor:surgeryislarge}. The lack of this inequality means that
the $d$--invariants of a surgery on $J$ can be calculated, but the formulae are much more complicated. In particular, in general
Theorem~\ref{thm:compute} does not hold.

\smallskip
The most important condition is that $b_2^+(X)=1$, which was rephrased as $p_g(X)=0$. Without this condition, the four-manifold $W$
has indefinite intersection form. Then Theorem~\ref{thm:osabs} does not hold. For a counterexample one can drill out a ball of $\C P^2$
to get a positive definite manifold $W$ with boundary $S^3$. By choosing an appropriate \spinc{} structure on $W$ we can make $c_1^2(\sst)$
arbitrary large and the $d$-invariant of the boundary ($S^3$ has unique \spinc{} structure) is equal to 0.

\smallskip
We present a variant of Theorem~\ref{thm:maintop} for cuspidal curves in smooth four-manifolds. The topological assumption, that $X$
admits an almost complex structure can be further relaxed, but at the cost of losing transparency.
\begin{theorem}\label{thm:maintop2}
Let $X$ be a smooth four-manifold with $b_2^+(X)=1$.
Suppose $C\subset X$ is a smooth surface with a finite number of singularities $z_1,\ldots,z_n$ which are cones of knots $J_1,\ldots,J_n$
such that $J_1,\ldots,J_n$ are all algebraic knots (or, more generally, L-space knots). Let $J=J_1\#\ldots\#J_n$. Suppose that $2g(J)<C^2$.
If $X$ admits an almost complex structure, then for any $E\in H_2(X;\Z)$ such that $C\cdot E\in[0,2g(J)]$ we have
\[R(C\cdot E+1)\ge \frac12 E\cdot(E+K),\]
where $K=c_1(T^*X)$ and $R$ is defined like in \eqref{eq:defofR}, with the remark that if $J_i$ is an L-space knot and its Alexander polynomial
is written as $\Delta_{J_i}(t)=1+(t-1)(t^{g_{i1}}+\ldots+t^{g_{ik_i}})$, then $S_i=\Z_{\ge 0}\setminus\{g_{i1},\ldots,g_{ik_i}\}$.
\end{theorem}
\begin{remark}\
\begin{itemize}
\item If $b_2^+(X)=1$, then the condition that $\pi_1(X)=e$ is sufficient to guarantee that $X$ has an almost complex structure; see \cite[Theorem 1.4.15 and
Exercise 1.4.16(b)]{GS}.
\item Formally, an L-space knot is a knot that admits a positive surgery with Heegaard Floer homologies behaving like Heegaard Floer homologies of a lens space;
see \cite{OSlspace}. By a result of Hedden, \cite{Hed}, this class contains all the algebraic knots.
\item For non-algebraic L-space knots, $S_i$ is not necessarily a semigroup, however with this definition of $S_i$ and $R$, Theorem~\ref{thm:compute}
works if $J$ is a connected sum of L-space knots; see \cite{BL2} and references therein.
\end{itemize}
\end{remark}
\begin{proof}[Sketch of proof of Theorem~\ref{thm:maintop2}]
We follow the proof of Theorem~\ref{thm:maintop}. The assumption that $C^2>2g(J)$ in Theorem~\ref{thm:maintop2} takes care
of Corollary~\ref{cor:surgeryislarge} in that proof, so we will be able to use Theorem~\ref{thm:compute}.
The analogue of Lemma~\ref{lem:Wisnegative} holds, because $b_2^+(X)=1$, so $b_2^+(W)=0$.

A small technical problem appears in  Lemma~\ref{lem:spinc}, because $g(J)$ is not necessarily equal to $\frac12 C\cdot(C+K)+1$.
Nevertheless, introduce the `virtual genus' $v_g=\frac12 C\cdot(C+K)+1$. Then Lemma~\ref{lem:spinc} works, when $D$ is any homology class
in $H_2(X;\Z)$ (usually we cannot speak of the divisors on an almost complex manifold), provided that
$g(J)$ is replaced by $v_g$. Continuing the proof of Theorem~\ref{thm:maintop} we arrive at \eqref{eq:r1} which holds provided $g$
is replaced by $v_g$. The Hirzebruch signature formula still holds; see Remark~\ref{rem:hirz} above. 
With this in mind we obtain \eqref{eq:onR} and simplify it to
\[R(C\cdot E+1)\ge \frac12 E\cdot (E-K)+1,\]
so the proof is finished.
\end{proof}


\begin{thebibliography}{XXX}

\bibitem{BH} W.~Barth, K.~Hulek, C.~Peters, A.~Van de Ven, 
\emph{Compact Complex Surfaces 2ed}, Ergeb. Math. Grenzgeb., vol. 4, Springer-Verlag, Berlin, 2004.

\bibitem{BCG} J.~Bodn\'ar, D.~Celoria, M.~Golla, \emph{Cuspidal curves and Heegaard Floer homology}, preprint, arxiv 1409.3282. 

\bibitem{BHL} M.~Borodzik, M.~Hedden, C.~Livingston, \emph{Plane algebraic curves of arbitrary genus via Heegaard Floer homology}, preprint,
arxiv 1409.2111.
 
\bibitem{BL} M.~Borodzik and C.~Livingstion, {\em Heegaard Floer homologies and rational cuspidal curves}, Forum of Math. Sigma, \textbf{2} (2014), e28.

\bibitem{BL2}  M. Borodzik and C. Livingston, {\it Semigroups, d-invariants and deformations of cuspidal singular points of plane curves}, preprint,
arxiv 1305.2868.   
\bibitem{BM} M.~Borodzik, T.~Moe, \emph{Topological obstructions for rational cuspidal curves in Hirzebruch surfaces},
preprint, arxiv:1410.4644, to appear in Michigan Math. J.


\bibitem{EN} D. Eisenbud, W. Neumann, \textit{Three-dimensional
link theory and invariants of plane curve singularities}, Annals Math.
Studies \textbf{110}, Princeton University Press, Princeton, 1985.

\bibitem{FLMN} J.~Fern\'andez de Bobadilla, I.~Luengo, A.~Melle-Hern\'andez, A.~N\'emethi, 
\emph{On rational cuspidal projective plane curves}, Proc. London Math. Soc. \textbf{92} (2006), no. 1, 99--138. 


\bibitem{GS}  R. Gompf and A. Stipsicz, {\it  4-manifolds and Kirby calculus} Graduate Studies in Mathematics, 20. American Mathematical Society, Providence, RI, 1999. 

\bibitem{GH}
P.~Griffiths, J.~Harris,\emph{Principles of algebraic geometry}, Wiley Classics Library. John Wiley \& Sons, Inc., New York, 1994.
 
\bibitem{Hed} M.~Hedden,
\emph{On knot Floer homology and cabling. II},
Int. Math. Res. Not.  2009, No. 12, 2248--2274.

\bibitem{KP} M.~Koras, K.~Palka, \emph{The Coolidge--Nagata conjecture},
preprint, arxiv:1502.07149.

\bibitem{Nem} A.~N\'emethi, \emph{Graded roots and singularities}, in: Singularities in geometry and topology, 394--463, World Sci. Publ., 
Hackensack, NJ, 2007.

\bibitem{OS-absolute} P.~Ozsv\'ath, Z.~Szab\'o, \emph{Absolutely graded Absolutely graded Floer homologies and intersection 
forms for four-manifolds with boundary},
Adv. Math. \textbf{173} (2003), 179--261.

\bibitem{OS-knot} P.~Ozsv\'ath, Z.~Szab\'o,
\emph{Holomorphic disks and knot invariants},
Adv. Math. \textbf{186} (2004),  58--116.


\bibitem{OSlspace} P.~Ozsv\'ath, Z.~Szab\'o, \emph{On knot Floer homology and lens space surgeries},
Topology \textbf{44} (2005),   1281--1300.

\bibitem{OS-introduction} P.~Ozsv\'ath, Z.~Szab\'o,  \textit{An introduction to Heegaard Floer homology}, in:  \textit{Floer homology, gauge theory, and
   low-dimensional topology}, 3--27, Clay Math. Proc., 5, Amer. Math. Soc., Providence, RI, 2006.

\bibitem{OS-introduction2} P.~Ozsv\'ath, Z.~Szab\'o,  \textit{Lectures on Heegaard Floer homology}, in:  \textit{Floer 
homology, gauge theory, and low-dimensional topology}, 29--70,
   Clay Math. Proc., 5, Amer. Math. Soc., Providence, RI, 2006.
\bibitem{Pal0} K.~Palka, \emph{Cuspidal curves, minimal models and Zaidenberg's finiteness conjecture},  Adv. Math. \textbf{267} (2014), 1--43. 

\bibitem{Pal} K.~Palka, \emph{The Coolidge-Nagata conjecture, part I}, preprint, arxiv:1405.5197.

\bibitem{Voi} C.~Voisin, \emph{On some problems of Kobayashi and Lang; algebraic approaches}, Current developments in mathematics, 2003, 53--125, 
Int. Press, Somerville, MA, 2003. 

\bibitem{Wall} C.~T.~C.~Wall, \textit{Singular Points of Plane Curves}
London Mathematical Society Student Texts, 63. Cambridge University Press, Cambridge, 2004.

\bibitem{Xu} G.~Xu,
\emph{Subvarieties of general hypersurfaces in projective space},
J. Diff. Geom. \textbf{39} (1994), no. 1, 139--172. 
\end{thebibliography}
\end{document}